\documentclass[12pt, a4]{amsart}
\usepackage{amscd, amsmath, amssymb}
\usepackage{fullpage}

\theoremstyle{plain}
\newtheorem{thm}{Theorem}[section]

\newtheorem{pro}[thm]{Proposition}
\theoremstyle{definition}

\newtheorem{ex}[thm]{Example}
\newtheorem{rem}[thm]{Remark}

\numberwithin{equation}{section}

\begin{document}

\title[Systems of perturbed Hammerstein integral equations]{Positive solutions of systems of perturbed Hammerstein integral equations with arbitrary order dependence}  

\date{}

\author[G. Infante]{Gennaro Infante}
\address{Gennaro Infante, Dipartimento di Matematica e Informatica, Universit\`{a} della
Calabria, 87036 Arcavacata di Rende, Cosenza, Italy}%
\email{gennaro.infante@unical.it}%

\begin{abstract}
Motivated by the study of systems of higher order boundary value problems with functional boundary conditions, we discuss, by topological methods, the solvability of a fairly general class of systems of perturbed Hammerstein integral equations, where the nonlinearities and the functionals involved depend on some derivatives. We improve and complement earlier results in the literature. We also provide some examples in order to illustrate the applicability of the theoretical results.
\end{abstract}



\maketitle

\section{Introduction}
In this paper we discuss the solvability of systems of perturbed Hammerstein integral equations of the form
\begin{multline}\label{perhamm-sys-intro}
u_{i}(t)=\lambda_{i} \int_0^1 k_{i}(t,s)f_{i}(s, u_{1}(s),\ldots, u_{i}^{(m_1)}(s), \ldots,  u_{n}(s), \ldots, u_{n}^{(m_n)}(s))\,ds\\
+ \sum_{j=1}^{p_i}\eta_{ij}{\gamma_{ij}}(t) h_{ij}[u], \ t\in [0,1],\ i=1,2,\ldots,n,\phantom{space}
\end{multline}
where $u=(u_{1},\ldots,  u_{n})$, the kernels $k_i$ are sufficiently regular, $f_i$ are continuous, $\gamma_{ij}$ are sufficiently smooth, $h_{ij}$ are compact functionals that are allowed to take into account higher order derivatives and   
$\lambda_{i}, \eta_{ij}$ are parameters. 

One motivation for studying the kind of equations that occur in~\eqref{perhamm-sys-intro} is that these often occur in applications, we refer the reader to the Introduction of~\cite{genupa} and the references therein. The case $n=1$ has been studied recently by Goodrich~\cite{Goodrich9, Goodrich10}, who complemented the earlier works~\cite{genupa, gi-caa}. In particular, Goodrich studied the equation
$$
u_{1}(t)=\lambda_{1} \int_0^1 k_{1}(t,s)f_{1}(s, u_{1}(s))\,ds
+ \sum_{j=1}^{2}{\gamma_{1j}}(t) h_{1j}[u_1],
$$
where the functionals $h_{1j}$ have the specific form
\begin{equation}\label{sform}
h_{1j}[u_1]=h_{j}(\alpha_{j}[u_1]).
\end{equation}
In~\eqref{sform} the functions $h_{j}$ are continuous and
$\alpha_{j}$ are \emph{linear} functionals on the space $C[0,1]$ which can be represented as Stieltjes integrals, namely
\begin{equation}\label{Stieltjes}
\alpha_{j}[u]:=\int_{0}^{1}u_1(s)\,dA_j(s). 
\end{equation} The functional formulation~\eqref{Stieltjes} is well suited for handling, in a unified way, multi-point and integral BCs. For an introduction to nonlocal BCs we refer the reader to the reviews~\cite{Cabada1, Conti, rma, Picone, sotiris, Stik, Whyburn} and the manuscripts~\cite{kttmna, ktejde, jw-gi-jlms}.

The case $n=2$ has been investigated in~\cite{genupa}, where the authors studied the system
\begin{equation*}
u_{i}(t)=\int_0^1 k_{i}(t,s)f_{i}(s, u_{1}(s),  u_{2}(s))\,ds
+ \sum_{j=1}^{2}{\gamma_{ij}}(t) h_{ij}[(u_1,u_2)],\ i=1,2,
\end{equation*}
where the functionals $h_{ij}$ act on the space $C[0,1]\times C[0,1]$. 

We stress that functionals involving higher order derivatives play an important role in applications. In order to illustrate this fact in a simple situation, consider the BVP
\begin{equation}\label{beam}
u_{1}^{(4)}(t)=f_{1}(t,u_{1}(t)),\quad u_{1}(0)=h_{11}[u_{1}],\ u_{1}^{\prime \prime }(0)= u_{1}(1)= u_{1}^{\prime \prime}(1)=0. 
\end{equation}
When $h_{11}[u_{1}] \equiv 0$ the BVP~\eqref{beam} can be used to describe the steady-state case of a simply supported beam of length 1. When the functional $h_{11}$ is non-trivial the BVP~\eqref{beam} can be used to model a beam with a feedback control; for example the case
\begin{equation}\label{shear}
h_{11}[u_{1}] =h(u_1'''(\xi)),
\end{equation}
models a beam with the right end simply supported and where the displacement in the left end is controlled (possibly in a nonlinear manner) by a sensor that measures the shear force in a point $\xi$ placed along of the beam. The perturbed integral equation associated to~\eqref{beam}-\eqref{shear} is
$$
u_{1}(t)= \int_0^1 k_{1}(t,s)f_{1}(s, u_{1}(s))\,ds
+ (1-t) h(u_{1}'''(\xi)),
$$
a case that cannot be handled with the theory developed in~\cite{genupa, Goodrich9, Goodrich10, gi-caa} due to the third order term occurring in~\eqref{shear}. 

The case of higher order dependence within the equation has been in investigated recently, by means of the classical Krasnosel'ski\u\i{}'s theorem of cone compression-expansion, by de Sousa and Minh\'os~\cite{rs-fm-narwa-19}. In particular, the authors of~\cite{rs-fm-narwa-19} consider the existence of nontrivial solutions for the system of Hammerstein equations 
\begin{equation*}
u_{i}(t)= \int_0^1 k_{i}(t,s)f_{i}(s, u_{1}(s),\ldots, u_{i}^{(m_1)}(s), \ldots,  u_{n}(s), \ldots, u_{n}^{(m_n)}(s))\,ds,\ i=1,2,\ldots,n.
\end{equation*}
As an interesting application of their theory, de Sousa and Minh\'os apply their result to a system of BVPs of the form
\begin{equation} \label{sys-intro-sm}
\left\{
\begin{array}{c}
u_{1}^{\prime \prime }(t)+ f_{1}(t,u_{1}(t), u_{1}'(t),u_{2}(t), u_{2}'(t), u_{2}''(t), u_{2}'''(t))=0,\quad t\in (0,1), \\
u_{2}^{(4)}(t)= f_{2}(t,u_{1}(t), u_{1}'(t),u_{2}(t), u_{2}'(t), u_{2}''(t), u_{2}'''(t)),\quad  t\in (0,1), \\
u_{1}(0)= u_{1}(1)=
u_{2}(0)=  u_{2}(1)= u_{2}^{\prime \prime }(0)=  u_{2}^{\prime \prime}(1)=0. %
\end{array}%
\right. 
\end{equation}
The system~\eqref{sys-intro} can be used as a model of the displacement of simply supported suspension bridge. In this model the fourth order equation describes the road bed and the second order equation models the suspending cables, we refer to~\cite{rs-fm-narwa-19} for more details. 

On the other hand, the case of equations of the form
$$
u_{1}(t)=\lambda_{1} \int_0^1 k_{1}(t,s)f_{1}(s, u_{1}(s), u_{1}'(s))\,ds
+ \sum_{j=1}^{2}\eta_{1j}{\gamma_{1j}}(t) h_{1j}[u_1],
$$
where the functionals $h_{1j}$ act on the space $C^{1}[0,1]$, has been studied recently by the author~\cite{gi-der}, by means of the classical fixed point index. Here we develop further this approach and we extend the results of~\cite{gi-der} to the case of systems and higher order dependence in the nonlinearities and the functionals. We also improve the case  $n=1$ and $m_1=1$, by allowing more freedom in the growth of the nonlinearities near the origin, this is achieved by means of an eigenvalue comparison.

In order to illustrate the applicability of our theory, we discuss, \emph{merely as an example}, the solvability of the system of the following model problem
\begin{equation} \label{sys-intro}
\left\{
\begin{array}{c}
u_{1}^{\prime \prime }(t)+\lambda_1 f_{1}(t,u_{1}(t), u_{1}'(t),u_{2}(t), u_{2}'(t), u_{2}''(t), u_{2}'''(t))=0,\quad t\in (0,1), \\
u_{2}^{(4)}(t)=\lambda_2 f_{2}(t,u_{1}(t), u_{1}'(t),u_{2}(t), u_{2}'(t), u_{2}''(t), u_{2}'''(t)),\quad  t\in (0,1), \\
u_{1}(0)=0,\ u_{1}(1)=\eta_{11}h_{11}[(u_{1},u_{2})], \\
u_{2}(0)=\eta_{21}h_{21}[(u_{1},u_{2})],\ u_{2}^{\prime \prime }(0)= u_{2}(1)= u_{2}^{\prime \prime}(1)=0, %
\end{array}%
\right. 
\end{equation}
where $h_{11}, h_{21}$ are nonnegative, compact functionals defined on the space $C^{1}[0,1]\times C^{3}[0,1]$. The interest in~\eqref{sys-intro} arises in the fact that it presents a coupling in the nonlinearities $f_1$ and $f_2$ and in the boundary conditions and allows the presence derivatives of different order in the various components. The system~\eqref{sys-intro} can be seen as a perturbation of the system~\eqref{sys-intro-sm} and is a generalisation of some earlier ones studied in~\cite{gifmpp-cnsns, gipp-mmas}. Here we discuss in detail the existence and non-existence of positive solutions of the system~\eqref{sys-intro}, illustrating how the constants that occur in our theory can be computed or estimated. Our results are new and complement the ones in~\cite{genupa, hen-luca, gi-caa, gi-der, kejde, rs-fm-narwa-19, ya1}.

\section{Main results}
In this Section we  study the existence and non-existence of solutions of the system of perturbed Hammerstein equationa of the type
\begin{multline}\label{perhamm-sys}
u_{i}(t)=\lambda_{i} \int_0^1 k_{i}(t,s)f_{i}(s, u_{1}(s),\ldots, u_{i}^{(m_1)}(s), \ldots,  u_{n}(s), \ldots, u_{n}^{(m_n)}(s))\,ds\\
+ \sum_{j=1}^{p_i}\eta_{ij}{\gamma_{ij}}(t) h_{ij}[u]:=T_{i}u(t) , \ t\in [0,1],\ i=1,2,\ldots,n,\phantom{space}
\end{multline}
where $u=(u_{1},\ldots,  u_{n}).$
Throughout the paper we make the following assumptions on the terms that occur in~\eqref{perhamm-sys}. 
\begin{itemize}
\item[$(C_1)$] For every $i=1,\ldots,n$,  $k_i:[0,1] \times[0,1]\rightarrow [0,+\infty)$ is 
measurable and continuous in $t$ for almost every  (a.e.)~$s$,
that is, for every $\tau\in [0,1] $ we have
\begin{equation*}
\lim_{t \to \tau} |k_i(t,s)-k_i(\tau,s)|=0 \ \text{for a.e.}\ s \in [0,1];
\end{equation*}
furthermore there exist a function $\Phi_{i0} \in L^{1}(0,1)$ such that
$0\leq  k_i(t,s)\leq \Phi_{i0}(s)$ for $t \in [0,1]$ and a.e.~$s\in [0,1]$.

\item[$(C_2)$] For every $i=1,\ldots,n$ and for every $l_{i} \in \mathbb{N}$, with $l_{i}<m_i$,  the partial derivative $\dfrac{\partial^{l_i}k_{i}}{\partial t^{l_i}}$ is measurable and continuous in $t$ for a.e.~$s$,
and there exists $\Phi_{il_{i}}(s) \in L^1(0, 1)$ such that $\Bigl|\dfrac{\partial^{l_i}k_{i}}{\partial t^{l_i}}(t,s)\Bigr|\le \Phi_{il_{i}}(s)$ for $t \in [0,1]$ and a.e. $s \in [0,1]$.

\item[$(C_3)$] For every $i=1,\ldots,n$, $\dfrac{\partial^{m_i}k_{i}}{\partial t^{m_i}}$ is measurable and continuous in $t$ except possibly at the point $t=s$ where there can be a jump discontinuity, that is, right and left limits both exist, 
and there exists $\Phi_{im_i}(s) \in L^1(0, 1)$ such that $\Bigl|\dfrac{\partial^{m_i}k_{i}}{\partial t^{m_i}}(t,s)\Bigr|\le \Phi_{im_i}(s)$ for $t \in [0,1]$ and a.e. $s \in [0,1]$.

\item[$(C_4)$] For every $i=1,\ldots,n$, $f_i:[0,1]\times \prod_{i=1}^{n} \bigl([0,+\infty)\times \mathbb{R}^{m_i}\bigr) \to [0,+\infty)$ is continuous.

\item[$(C_5)$] For every $i=1,\ldots,n$ and $j=1,\dots, p_i$, we have $\gamma_{ij} \in C^{m_i} [0,1] $ and $\gamma_{ij}(t)\geq 0\ \text{for every}\ t\in [0,1]$.

\item[$(C_6)$] For every $i=1,\ldots,n$ and $j=1,\dots, p_i$, we have $\lambda_i, \eta_{ij},  \in [0,+\infty)$.
\end{itemize}
Due to the assumptions above, for every $i=1,\ldots,n$, the linear Hammerstein integral operator
$$
L_{i}w(t):=\int_0^1 k_{i}(t,s) w(s) \,ds
$$
is well defined and compact in the space $C[0,1]$, where we adopt the standard norm $\|w\|_{\infty}:=\max_{t\in [0,1]}|w(t)|$. We recall that a cone $K$ in a real Banach space $X$ is a closed convex set such that $\lambda x\in K$ for every $x \in K$ and for all $\lambda\geq 0$ and satisfying $K\cap (-K)=\{0\}$. It is clear that the operator $L_{i}$ leaves invariant the cone 
$$
\hat{P}:=\{w\in C[0,1]:\ w\ge 0\ \text{for every}\ t\in [0,1] \}.
$$
We denote by $r(L_{i})$ the spectral radius of $L_{i}$ and assume
\begin{itemize}
\item[$(C_7)$] For every $i=1,\ldots,n$, we have $r(L_{i})>0$.
\end{itemize}
Note that, since $\hat{P}$ is a reproducing cone in $C[0,1]$, the assumption $(C_7)$ allows us to apply the well-know Krein-Rutman Theorem and therefore $r(L_{i})$ is an
eigenvalue of $L_{i}$ with a corresponding eigenfunction $\varphi_i\in \hat{P}\setminus \{0\}$, that is
\begin{equation}\label{eigenfunction}
L_{i}\varphi_i(t)=\int_0^1 k_{i}(t,s) \varphi_i (s) \,ds=r(L_{i}) \varphi_i(t).
\end{equation}
In what follows we shall make use of the eigenfunction $\varphi_i$ and the corresponding characteristic value 
\begin{equation*}
\mu_i:=1/r(L_{i}).
\end{equation*}
Note that the non-negative eigenfunction $\varphi_i$ inherits, from the kernel $k_i$, further regularity properties: indeed, since we have
\begin{equation}\label{reg}
\varphi_i (t)=\mu_i \int_0^1 k_{i}(t,s) \varphi_i (s) \,ds,
\end{equation}
and, due to the assumptions $(C_1)$-$(C_3)$, the RHS of~\eqref{reg} is, as a function of the variable $t$, in $C^{m_i}[0,1]$ we obtain 
$$\varphi_i \in (\hat{P}\setminus \{0\}) \cap C^{m_i}[0,1].$$

\begin{rem}
The assumption $(C_7)$ is frequently satisfied in applications. A sufficient condition, for details see~\cite{jw-kql-tmna}, is given by
\begin{itemize}
\item [$(C'_7)$]
There exist a subinterval $[a_i,b_i] \subseteq [0,1]$ and a constant $c_i=c(a_i,b_i) \in (0,1]$ such that
$$
k_i(t,s) \geq c_i\Phi_{i0}(s) \text{ for } t\in [a_i,b_i] \text{ and a.\,e. } \, s \in [0,1].
$$
\end{itemize}
\end{rem}

Due to the hypotheses above, we work in the product space 
$\displaystyle \prod_{i=1}^{n} C^{m_i}[0,1]$ endowed with the norm 
$$\|u\|:=\max_{i=1,\ldots,n}\{\|u_i\|_{C^{m_i}}\},$$
where $\|u_i\|_{C^{m_i}}:=\displaystyle \max_{j=0,\ldots,m_i}\{\|u_i^{(j)}\|_\infty\}$. We utilize the cone 
$$
P:=\Bigl\{u\in \prod_{i=1}^{n} C^{m_i}[0,1]:\ u_i\ge 0\ \text{for every}\ t\in [0,1],\ i=1,\dots, n\Bigl\}.
$$
and we require  the nonlinear functionals $h_{ij}$ to act positively on the cone $P$ and to be compact, that is:
\begin{itemize}
\item[$(C_8)$] For every $i=1,\ldots,n$ and $j=1,\dots, p_i$, $h_{ij}: P \to [0,+\infty)$ is continuous and map bounded sets into bounded sets.
\end{itemize}
We define the operator $T: P \to P$ as
\begin{equation}
T u:=\bigl(T_i u\bigr)_{i=1\ldots n}.
\end{equation}

We make use of the following basic properties of the fixed point index, we refer the reader to~\cite{amann, guolak} for more details.

\begin{pro}\cite{amann, guolak} Let $K$ be a cone in a real Banach space $X$ and let
$D$ be an open bounded set of $X$ with $0 \in D_{K}$ and
$\overline{D}_{K}\ne K$, where $D_{K}=D\cap K$.
Assume that $\tilde{T}:\overline{D}_{K}\to K$ is a compact map such that
$x\neq \tilde{T}x$ for $x\in \partial D_{K}$. Then the fixed point index
 $i_{K}(\tilde{T}, D_{K})$ has the following properties:
\begin{itemize}
\item[$(1)$] If there exists $e\in K\setminus \{0\}$
such that $x\neq \tilde{T}x+\lambda e$ for all $x\in \partial D_K$ and all
$\lambda>0$, then $i_{K}(\tilde{T}, D_{K})=0$.

\item[$(2)$] If $\tilde{T}x \neq \lambda x$ for all $x\in
\partial D_K$ and all $\lambda > 1$, then $i_{K}(\tilde{T}, D_{K})=1$.

\item[(3)] Let $D^{1}$ be open in $X$ such that
$\overline{D^{1}}_{K}\subset D_K$. If $i_{K}(\tilde{T}, D_{K})=1$ and $i_{K}(\tilde{T},
D_{K}^{1})=0$, then $\tilde{T}$ has a fixed point in $D_{K}\setminus
\overline{D_{K}^{1}}$. The same holds if 
$i_{K}(\tilde{T}, D_{K})=0$ and $i_{K}(\tilde{T}, D_{K}^{1})=1$.
\end{itemize}
\end{pro}
For $\rho\in (0,\infty)$, we define the sets
$$
P_{\rho}:=\{u\in P: \|u\|<\rho\},\quad I_{\rho}:=[0,1]\times \prod_{i=1}^{n} \bigl([0,\rho]\times [-\rho, \rho]^{m_i}\bigr)
$$
and the quantities
$$
\overline{f}_{i\rho}:=\max_{I_{\rho}}f_i(t, x_{10},\ldots, x_{1m_1},\ldots, x_{n0},\ldots, x_{nm_n}),\quad H_{ij\rho}:=\sup_{u\in \partial P_{\rho}} h_{ij}[u],
$$
$$
K_{il}:=\begin{cases}\displaystyle \sup_{t\in[0,1]}\int_0^1 k_i(t,s)\,ds,\ l=0,\\
\displaystyle \sup_{t\in[0,1]}\int_0^1 \Bigl|\frac{\partial^{l}k_{i}}{\partial t^{l}}(t,s)\Bigr|\,ds,\ l=1,\ldots, m_i. 
\end{cases}
$$

With these ingredients we can state the following existence and localization result.
\begin{thm}\label{thmsol}
Assume there exist $r,R,\delta\in (0,+\infty)$, with $r<R$, and $i_0\in \{1,2,\ldots,n\}$ such that the following three inequalities are satisfied:
\begin{equation}\label{idx1}
\max_{\substack{i=1,\ldots, n\\ l=0,\ldots, m_i}}\Bigl\{\lambda_{i}\overline{f}_{iR} K_{il}+\sum_{j=1}^{p_i}\eta_{ij}\|{\gamma_{ij}^{(l)}}\|_{\infty}H_{ijR} \Bigr \}\leq R,
\end{equation}
\begin{equation}\label{idx0}
\lambda_{i_0}\geq \frac{\mu_{i_{0}}}{\delta},\quad f_{i_{0}}(t, x_{10},\ldots, x_{1m_1},\ldots, x_{n0},\ldots, x_{nm_n})\ge \delta x_{i_{0} 0},\ \text{on}\  I_{r}.
\end{equation}

Then the system~\eqref{perhamm-sys} has a solution $u\in P$ such that $$r\leq \|u\| \leq R.$$
\end{thm}

\begin{proof}
With a careful use of the Ascoli-Arzel\`{a} theorem, it is can be proved that, under the assumptions $(C_{1})$-$(C_{8})$, the operator $T$ maps $P$ into $P$ and is compact.

If $T$ has a fixed point either on $\partial {P_r}$ or $\partial {P_R}$ we are done. 
Assume now that $T$ is fixed point free on $\partial {P_r}\cup\partial {P_R}$, we are going to prove that $T$ has a fixed point in 
$ P_R\setminus \overline {P_r}$.

We firstly prove that 
 $
\sigma  u\neq Tu\ \text{for every}\ u\in \partial P_{R}\
\text{and every}\  \sigma >1.
$
If this does not hold, then there exist $u\in \partial P_{R}$ and $\sigma >1$ such that $\sigma  u= Tu$. 
Note that if $\| u\| = R$ then there exist $i_0, l_0$ such that 
$\| u_{i_0}^{(l_0)}\|_{\infty} = R$. We show the case $l_0\neq 0$ (the case $l_0=0$ is simpler, hence omitted)
Thus we have, for $t\in [0,1]$,
\begin{equation}\label{ineq1}
\begin{aligned}
\sigma u_{i_0}^{(l_0)}(t)=&\lambda_{i} \int_0^1 \frac{\partial^{l_0}k_{i_0}}{\partial t^{l_0}}(t,s)f_{i_0}(s, u_{1}(s),\ldots,u_{i}^{(m_1)}(s), \ldots,  u_{n}(s), \ldots, u_{n}^{(m_n)}(s))\,ds\\
&+ \sum_{j=1}^{p_i}\eta_{i_0j}{\gamma_{i_0j}^{(l_0)}}(t) h_{i_0j}[u].
\end{aligned}
\end{equation}
From~\eqref{ineq1} we obtain, for $t\in [0,1]$,
\begin{multline}\label{ineq2}
\sigma |u_{i_0}^{(l_0)}(t)|\leq \lambda_{i} \int_0^1 \Bigl| \frac{\partial^{l_0}k_{i_0}}{\partial t^{l_0}}(t,s)\Bigr| f_{i_0}(s, u_{1}(s),\ldots,u_{i}^{(m_1)}(s), \ldots,  u_{n}(s), \ldots, u_{n}^{(m_n)}(s))\,ds\\
+ \sum_{j=1}^{p_i}\eta_{i_0j}|{\gamma_{i_0j}^{(l_0)}}(t)| h_{i_0j}[u]\leq \lambda_{i_0}\overline{f}_{i_0R} K_{i_0l_0}+\sum_{j=1}^{p_i}\eta_{i_0j}\|{\gamma_{i_0j}^{(l_0)}}\|_{\infty}H_{i_0jR}\leq R.
\end{multline}
Taking in~\eqref{ineq2} the supremum for $t\in [0,1]$ yields  $\sigma \leq 1$, a contradiction.\newline
Therefore we have $i_{P}(T, P_R)=1.$

We now consider the function $\varphi(t):= (\varphi_1(t),\ldots,\varphi_n(t))$, where $t\in [0,1]$ and $\varphi_i$ is given by~\eqref{eigenfunction}. Note that $\varphi\in P\setminus \{0\}$. We show that 
\begin{equation*} 
u\neq Tu+\sigma \varphi\ \text{for every}\ u\in \partial P_{r}\ \text{and every}\ \sigma  >0.
\end{equation*}
If not, there exists $u\in \partial P_{r}$ and $\sigma  >0$ such that
$
u= Tu+\sigma  \varphi . 
$
In particular, we have $
u_{i_0}(t)= T_{i_0}u(t)+\sigma  \varphi_{i_0} (t) 
$ for every $t\in [0,1]$ and therefore $u_{i_0}(t)\geq \sigma  \varphi_{i_0}(t)$ in $ [0,1]$.
Observe that we have $r\geq \|u_{i_0}\|_{\infty}\geq \sigma \|  \varphi_{i_0} \|_{\infty}>0$.\newline
 For every $t\in [0,1]$ we have
\begin{multline*}
u_{i_0}(t)=\lambda_{i} \int_0^1 k_{i_0} (t,s) f_{i_0}(s, u_{1}(s),\ldots,u_{i}^{(m_1)}(s), \ldots,  u_{n}(s), \ldots, u_{n}^{(m_n)}(s))\,ds\\
+ \sum_{j=1}^{p_i}\eta_{i_0j}\gamma_{i_0j}(t) h_{i_0j}[u] + \sigma \varphi_{i_0}(t)  
\ge  \lambda_{i} \int_0^1 k_{i_0} (t,s) \delta u_{i_0}(s)\,ds + \sigma \varphi_{i_0}(t) \\
\ge  \lambda_{i}  \int_0^1 k_{i_0} (t,s) \delta \sigma  \varphi_{i_0}(s)\,ds + \sigma \varphi_{i_0}(t)\\ 
=  \frac{ \sigma  \lambda_{i_0} \delta}{\mu_{i_{0}}} \varphi_{i_0}  (t) +\sigma  \varphi_{i_0}(t)  \geq 2\sigma  \varphi_{i_0}(t).
\end{multline*}
By iteration we obtain, for $t\in [0,1]$,
$$
u_{i_0}(t)\ge n\sigma  \varphi_{i_0}(t) \ \text{for every}\ n\in\mathbb{N},
$$
which contradicts the fact that $\|u_{i_0}\|_{\infty}\leq r$.\newline
Thus we obtain $i_{P}(T, P_{r})=0.$

Therefore we have
$$i_{P}(T, P_R \setminus \overline{P_r})=i_{P}(T, P_R )-i_{P}(T, P_r )=1,$$
which proves the result.
\end{proof}
We now illustrate the applicability of Theorem~\ref{thmsol}.
\begin{ex}
We focus on the system
\begin{equation} \label{bvpex1}
\left\{
\begin{array}{c}
u_{1}^{\prime \prime }(t)+\lambda_1 f_{1}(t,u_{1}(t), u_{1}'(t),u_{2}(t), u_{2}'(t), u_{2}''(t), u_{2}'''(t))=0,\quad t\in (0,1), \\
u_{2}^{(4)}(t)=\lambda_2 f_{2}(t,u_{1}(t), u_{1}'(t),u_{2}(t), u_{2}'(t), u_{2}''(t), u_{2}'''(t)),\quad  t\in (0,1), \\
u_{1}(0)=0,\ u_{1}(1)=\eta_{11}h_{11}[(u_{1},u_{2})], \\
u_{2}(0)=\eta_{21}h_{21}[(u_{1},u_{2})],\ u_{2}^{\prime \prime }(0)= u_{2}(1)= u_{2}^{\prime \prime}(1)=0, %
\end{array}%
\right. 
\end{equation}
where $h_{11}, h_{21}$ are nonnegative, compact functionals acting on the cone
$$
P=\bigl\{(u_{1},u_{2})\in C^{1}[0,1]\times C^{3}[0,1]:\ u_{1}, u_{2}\ge 0\ \text{for every}\ t\in [0,1]\bigl\}.
$$
With our methodology we could study a more complicated version of this BVP, by adding more functional terms in the BCs, but we refrain from doing so for the sake of clarity.

It is routine to show that the solutions of~\eqref{bvpex1} can be written in the form 
\begin{gather}
\left\{
\begin{aligned}\label{perturbed-ex}
u_{1}(t)=&\eta_{11} t h_{11}[(u_{1},u_{2})]\\&+\lambda_1\int_{0}^{1}k_1(t,s)f_{1}(s,u_{1}(s), u_{1}'(s),u_{2}(s), u_{2}'(s), u_{2}''(s), u_{2}'''(s))\,ds, \\
u_{2}(t)= &\eta_{21}(1-t) h_{21}[(u_{1},u_{2})]\\&+\lambda_2\int_{0}^{1}k_2(t,s)f_{2}(s,u_{1}(s), u_{1}'(s),u_{2}(s), u_{2}'(s), u_{2}''(s), u_{2}'''(s))\,ds,%
\end{aligned}
\right. 
\end{gather}
where
\begin{equation}\label{kernels}
k_1(t,s)=%
\begin{cases}
s(1-t),\, & s \leq t, \\
t(1-s),\, & s>t,%
\end{cases}
\quad \text{and}\quad k_2(t,s)=
\begin{cases}
\frac{1}{6} s(1-t)(2t-s^{2}-t^{2}),\,  & s\leq t, \\
\frac{1}{6} t(1-s)(2s-t^{2}-s^{2}),\,  & s>t.%
\end{cases}%
\end{equation}
It is known that the kernels $k_1$ and $k_2$ that occur in~\eqref{kernels} are continuous, non-negative, satisfy condition $(C_7)$ and (see for example~\cite{kor-88, jw-gi-df, jw-kql-tmna})
\begin{equation*}
K_{10}=\frac{1}{8},\ \mu_1=\pi^{2},\ K_{20}=\frac{5}{384},\ \mu_2=\pi^{4}.
\end{equation*}
By direct calculation we obtain 
\begin{equation*}
\frac{\partial k_1}{\partial t}(t,s)=%
\begin{cases}
-s,\, & s < t, \\
(1-s),\, & s>t,%
\end{cases}
\quad
\frac{\partial k_2}{\partial t}(t,s)=\frac{1}{6}
\begin{cases}
s(-6t+s^2+3t^2+2),\,  & s\leq t, \\
(1-s)(-s^2+2s-3t^2),\,  & s>t,%
\end{cases}%
\end{equation*}
\begin{equation*}
\frac{\partial^2 k_2}{\partial t^2}(t,s)=
\begin{cases}
s(t-1),\,  & s\leq t, \\
t(s-1),\,  & s>t.%
\end{cases}%
\quad \text{and}\quad 
\frac{\partial^3 k_2}{\partial t^3}(t,s)=
\begin{cases}
s,\,  & s< t, \\
(s-1),\,  & s>t.%
\end{cases}%
\end{equation*}
We may use 
\begin{equation*}
\Phi _{20}(s) =
\begin{cases}
\frac{\sqrt{3}}{27}s (1-s^{2})^{\frac{3}{2}},&
\;\; \text{for} \;\; 0\leq s\leq \frac{1}{2},\\
\frac{\sqrt{3}}{27}(1-s) s^{\frac{3}{2}}(2-s)^{\frac{3}{2}},& \;\;
\text{for} \;\; \frac{1}{2} < s\leq 1,
\end{cases}\quad \text{(see \cite{jw-gi-df}),}
\end{equation*}
\begin{equation*}
\Phi _{21}(s)=\frac{1}{6}s(2+s^2), \quad \text{(see~\cite{rs-fm-narwa-19}),}
\end{equation*}
and, by direct calculation, we take
\begin{equation*}
\Phi _{10}(s)=\Phi _{22}(s)=s(1-s),\quad \Phi _{11}(s)=\Phi _{23}(s)=\bigl|s-\frac{1}{2}\bigr|+\frac{1}{2}.
\end{equation*}
Therefore the assumptions $(C_1)$-$(C_3)$ are satisfied.
By direct computation we obtain
\begin{equation*}
K_{22}=\frac{1}{8},\ K_{11}=K_{23}=\frac{1}{2}, \ K_{21}\leq \frac{5}{24}.
\end{equation*}
Note that we have 
\begin{equation*}
{\gamma_{11}}(t)=t,\ {\gamma_{11}}'(t)=1,\ {\gamma_{21}}(t)=(1-t),\ {\gamma_{21}}'(t)=-1,
\end{equation*}
and therefore we get
\begin{equation*}
\|\gamma_{11}\|_{\infty}=\|\gamma_{11}'\|_{\infty}=\|{\gamma_{21}}\|_{\infty}=\|{\gamma_{21}}'\|_{\infty}=1,\ \|{\gamma_{21}}''\|_{\infty}=\|{\gamma_{21}}'''\|_{\infty}=0,
\end{equation*}
Thus the condition~\eqref{idx1} is satisfied if
\begin{equation}\label{idx1suf}
\max\Bigl\{\frac{1}{2}\lambda_{1}\overline{f}_{1R}+\eta_{11}H_{11R},\ \frac{5}{24}\lambda_{2}\overline{f}_{2R}+\eta_{21}H_{21R}, \frac{1}{2}\lambda_{2}\overline{f}_{2R} \Bigr \}\leq R.
\end{equation}
Let us now fix the nonlinearities $f_i$ and the functionals $h_{i1}$, say
\begin{multline*}
f_{1}(t,u_{1}(t), u_{1}'(t),u_{2}(t), u_{2}'(t), u_{2}''(t), u_{2}'''(t))=u_1^2(t)(2- t\sin(u_1'(t)+u_2''(t)),\\
f_{2}(t,u_{1}(t), u_{1}'(t),u_{2}(t), u_{2}'(t), u_{2}''(t), u_{2}'''(t))=\sqrt{u_2(t)} e^{t(u_1(t)+u_2'''(t))},\\
h_{11}[(u_{1},u_{2})]=\int_{0}^1 (u_1'(t)+u_2'''(t))^2\,dt,\\
h_{21}[(u_{1},u_{2})]=(u_1'(1/4))^2+(u_2''(3/4))^4,%
\end{multline*}
and prove the existence of solutions in $u\in P$ with $\|u\|\leq 1$. Thus we fix $R=1$.
Since $\bar{f}_{11}\leq 3,\ \bar{f}_{21}\leq e^2,\  H_{111}\leq 4,\ H_{211}\ \leq 2$, the condition~\eqref{idx1suf} is satisfied if the inequality
\begin{equation}\label{idx1suf2}
\max\Bigl\{\frac{3}{2}\lambda_{1}+4\eta_{11},\ \frac{5e^2}{24}\lambda_{2}+2\eta_{21}, \frac{e^2}{2}\lambda_{2} \Bigr \}\leq 1
\end{equation}
holds.
Note that $f_2$ satisfies condition~\eqref{idx0} for every fixed $\lambda_2 >0$, by choosing $r$ sufficiently small.
Therefore, for the range of parameters that satisfy the inequality~\eqref{idx1suf2} with $\lambda_2>0$, Theorem~\ref{thmsol} provides the existence 
of a solution of the system~\eqref{perturbed-ex} in $P$, with $0<\|u\| \leq 1$; this occurs, for example, for $\lambda_1=1/10, \lambda_2=1/5, \eta_{11}=1/5, \eta_{22}=1/3$.
\end{ex}
We now use an elementary argument to prove a non-existence result.
\begin{thm}\label{nonexthm}
Assume that there exist $\tau_i, \xi_{ij}\in (0,+\infty)$ such that
\begin{equation*}
0\leq f_{i} (t,x_{10},\ldots, x_{1m_1},\ldots, x_{n0},\ldots, x_{nm_n})\leq \tau_i x_{i0},\ \text{on}\ [0,1]\times \prod_{i=1}^{n} \bigl([0,+\infty)\times \mathbb{R}^{m_i}\bigr),
\end{equation*}
\begin{equation*}
h_{ij}[u]\leq \xi_{ij}\|u_i\|_{\infty}, \ \text{for every}\ u \in P,\ i=1\ldots n,\ j=1\ldots p_i,
\end{equation*}
\begin{equation}\label{nonexineq}
\max_{i=1,\ldots, n}\Bigl\{\lambda_i \tau_i K_{i0}+\sum_{j=1}^{p_i}\eta_{ij} \xi_{ij}\|{\gamma_{ij}}\|_{\infty}\Bigr\}<1.
 \end{equation}
Then the system~\eqref{perhamm-sys} has at most the zero solution in $P$. 
\end{thm}
\begin{proof}
Assume that there exist $u\in P\setminus \{0\}$ such that $Tu=u$. Then there exists $i_{0}\in \{1,\ldots, n\}$ such that $\|u_{i_{0}}\|_{\infty}=\rho$, for some $\rho>0$.
Then, 
for every $t\in [0,1]$, we have
\begin{equation}~\label{ineq5}
\begin{aligned}
u_{i_0}(t)=&\lambda_{i_0} \int_0^1 k_{i_0} (t,s) f_{i_0}(s, u_{1}(s),\ldots,u_{i}^{(m_1)}(s), \ldots,  u_{n}(s), \ldots, u_{n}^{(m_n)}(s))\,ds\\
&+ \sum_{j=1}^{p_i}\eta_{i_0j}\gamma_{i_0j}(t) h_{i_0j}[u] \\ 
\leq & \lambda_{i_0} \int_0^1 k_{i_0} (t,s) \tau_{i_0} u_{i0}\,ds+ \sum_{j=1}^{p_i}\eta_{i_0j}\gamma_{i_0j}(t) h_{i_0j}[u]\\
\leq & \lambda_{i_0} \int_0^1 k_{i_0} (t,s) \tau_{i_0} \rho\,ds+ \sum_{j=1}^{p_i}\eta_{i_0j}\gamma_{i_0j}(t) \xi_{i_0j}\rho\\
\leq & \lambda_{i_0} \tau_{i_0} K_{i_{0}0}\rho+ \sum_{j=1}^{p_i}\eta_{i_0j}\|\gamma_{i_0j}\|_{\infty} \xi_{i_0j}\rho.
\end{aligned}
\end{equation}
Taking the supremum for $t\in [0,1]$ in~\eqref{ineq5} gives $\rho < \rho$, a contradiction.
\end{proof}
We conclude by illustrating the applicability of Theorem~\ref{nonexthm}.
\begin{ex}
Let us now consider the system
\begin{equation} \label{bvpex2}
\left\{
\begin{array}{c}
u_{1}^{\prime \prime }(t)+\lambda_1  u_1(t)(2-t\sin (u_2(t)u_1'(t)))=0,\quad t\in (0,1), \\
u_{2}^{(4)}(t)=\lambda_2  u_2(t)(2-t\cos(u_1(t)u_2'''(t))),\quad  t\in (0,1), \\
u_{1}(0)=0,\ u_{1}(1)=\eta_{11} u_1(1/4)\cos^2 (u_1'(3/4)u_2''(1/4)), \\
u_{2}(0)=\eta_{21}u_{2}(3/4)\sin^2 (u_1'(1/4)u_{2}'''(3/4)),\ u_{2}^{\prime \prime }(0)= u_{2}(1)= u_{2}^{\prime \prime}(1)=0. %
\end{array}%
\right. 
\end{equation}
In this case we may take $\tau_1=\tau_2=3, \xi_{11}=\xi_{21}=1$.  Then the condition~\eqref{nonexineq} reads
\begin{equation}\label{bvpex2rg}
\max\Bigl\{\frac{3}{8}\lambda_{1}+\eta_{11},\ \frac{5}{128}\lambda_{2}+\eta_{21} \Bigr \}< 1.
\end{equation}

Since $(0,0)$ is a solution of the system~\eqref{bvpex2}, for the range of parameters that satisfy the inequality~\eqref{bvpex2rg}, Theorem~\ref{nonexthm} guarantees that the only possible solution in $P$ of the BVP~\eqref{bvpex2} is the trivial one; this occurs, for example, for $\lambda_1=1, \lambda_2=5, \eta_{11}=1/2, \eta_{22}=1/3$.
\end{ex}


\begin{thebibliography}{xxx}

\bibitem{amann} H. Amann, Fixed point equations and nonlinear
eigenvalue problems in ordered Banach spaces, \textit{SIAM. Rev.},
\textbf{18} (1976), 620--709.
%
\bibitem{Cabada1} A. Cabada, An overview of the lower and upper solutions method with nonlinear boundary value conditions, \textit{Bound. Value Probl.} (2011), Art. ID 893753, 18 pp.

\bibitem{genupa}
F. Cianciaruso, G. Infante and P. Pietramala, Solutions of perturbed Hammerstein integral equations with applications, \textit{Nonlinear Anal. Real World Appl.}, \textbf{33} (2017), 317--347.

\bibitem{Conti} R. Conti,
Recent trends in the theory of boundary value problems for ordinary differential equations,
\textit{Boll. Un. Mat. Ital.}, \textbf{22} (1967), 135--178.

\bibitem{Goodrich9}
C. S. Goodrich, New Harnack inequalities and existence theorems for radially symmetric solutions of elliptic PDEs with sign changing or vanishing Green's function,
\textit{J. Differential Equations}, \textbf{264} (2018), 236--262. 

\bibitem{Goodrich10}
C. S. Goodrich, Coercive functionals and their relationship to multiplicity of solution to nonlocal boundary value problems,
\textit{Topol. Methods Nonlinear Anal.}, to appear. 

%
\bibitem{guolak} D. Guo and V. Lakshmikantham,
\textit{Nonlinear problems in abstract cones}, Academic Press, Boston,
1988.
%

\bibitem{hen-luca}
J. Henderson and R. Luca, \textit{Boundary Value Problems for Systems of Differential, Difference and Fractional Equations. Positive Solutions}, Elsevier, Amsterdam, 2016.

\bibitem{gi-caa} G. Infante,
 Nonlocal boundary value problems with two nonlinear boundary conditions,
\textit{Commun. Appl. Anal.}, \textbf{12} (2008), 279--288.

\bibitem{gi-der} G. Infante,
Positive and increasing solutions of perturbed Hammerstein integral equations with derivative dependence, arXiv:1903.10900 [math.AP].

\bibitem{gifmpp-cnsns} G. Infante, F. M. Minh\'os and P. Pietramala,
Non-negative solutions of systems of ODEs with coupled boundary
conditions, \textit{ Commun. Nonlinear Sci. Numer. Simul.},
\textbf{17} (2012), 4952--4960.

\bibitem{gipp-mmas} G. Infante and P. Pietramala,
 Multiple nonnegative solutions of systems with coupled nonlinear boundary conditions,
 \textit{Math. Methods Appl. Sci.}, \textbf{37} (2014), 2080--2090.

\bibitem{kejde}
G. L. Karakostas,
Existence of solutions for an $n$-dimensional operator equation and applications to BVPs,
\textit{Electron. J. Differential Equations}, \textbf{2014}, No. 71, 17~pp.

\bibitem{kttmna} G. L. Karakostas and P. Ch. Tsamatos, Existence of multiple
positive solutions for a nonlocal boundary value problem,
\textit{Topol. Methods Nonlinear Anal.}, \textbf{19} (2002),
109--121.
%
\bibitem{ktejde}
G. L. Karakostas and P. Ch. Tsamatos, Multiple positive solutions
of some Fredholm integral equations arisen from nonlocal
boundary-value problems, \textit{Electron. J. Differential
Equations}, \textbf{2002}, 17 pp.

\bibitem{kor-88} P. Korman,
Computation of displacements for nonlinear elastic beam models using monotone iterations,
\textit{Int. J. Math. Math. Sci.}, \textbf{11} (1988), 121--128.

\bibitem{rma}
R. Ma, A survey on nonlocal boundary value problems, \textit{Appl.
Math. E-Notes}, \textbf{7} (2007), 257--279.

\bibitem{sotiris}
S. K. Ntouyas, Nonlocal initial and boundary value problems: a
survey, \textit{Handbook of differential equations: ordinary
differential equations. Vol. II}, Elsevier B. V., Amsterdam,
(2005), 461--557.

\bibitem{Picone} M. Picone, Su un problema al contorno nelle equazioni differenziali lineari ordinarie del secondo ordine,
\textit{Ann. Scuola Norm. Sup. Pisa Cl. Sci.}, \textbf{10} (1908), 1--95.

\bibitem{rs-fm-narwa-19} R. de Sousa and F. Minh\'os, Coupled systems of Hammerstein-type integral equations with sign-changing kernels,
\textit{Nonlinear Anal. Real World Appl.}, \textbf{50} (2019), 469--483.

\bibitem{Stik} A. \v{S}tikonas, A survey on stationary problems, Green's functions and spectrum of Sturm-Liouville problem
with nonlocal boundary conditions, \textit{Nonlinear Anal. Model.
Control}, \textbf{19} (2014), 301--334.

\bibitem{jw-gi-jlms} J. R. L. Webb and G. Infante,
Positive solutions of nonlocal boundary value problems: a unified
approach, \textit{J. London Math. Soc.}, \textbf{74} (2006),
673--693.

\bibitem{jw-gi-df} J. R. L. Webb, G. Infante and D. Franco, Positive solutions
of nonlinear fourth order boundary value
problems with local and nonlocal  boundary conditions,
\textit{Proc. Roy. Soc. Edinburgh Sect. A}, \textbf{138} (2008), 427--446.

\bibitem{jw-kql-tmna} J. R. L. Webb and K. Q. Lan, Eigenvalue criteria for
existence of multiple positive solutions of nonlinear boundary
value problems of local and nonlocal type, \textit{Topol. Methods Nonlinear Anal.},
\textbf{27} (2006), 91--116.

\bibitem{Whyburn}
W. M. Whyburn, Differential equations with general boundary
conditions, \textit{Bull. Amer. Math. Soc.}, \textbf{48} (1942),
692--704.

\bibitem{ya1} Z. Yang,
Positive solutions to a system of second-order nonlocal boundary
value problems, \textit{Nonlinear Anal.}, \textbf{62} (2005),
1251--1265.

\end{thebibliography}
\end{document}